\documentclass[a4paper,twoside]{amsart}
\usepackage{amssymb} 
\usepackage[all]{xy} 
\usepackage{color} 
\usepackage{lastpage} 

\newtheorem{thm}{Theorem}[section]

\newtheorem{lem}[thm]{Lemma}

\newtheorem{prop}[thm]{Proposition}

\theoremstyle{definition}
\newtheorem{defn}[thm]{Definition}
\theoremstyle{remark}
\newtheorem{rem}[thm]{Remark}
\newtheorem{exmpl}[thm]{Example}
\numberwithin{equation}{section}
\newcommand*{\cat}[1]{\mathbf{#1}} 
\newcommand*{\mto}{\rightarrow} 
\newcommand*{\cto}{\rightarrowtail} 
\newcommand*{\qto}{\twoheadrightarrow} 
\newcommand*{\wto}{\xrightarrow{\sim}} 
\newcommand*{\set}[1]{\left\{#1\right\}} 
\newcommand*{\del}{\partial} 
\newcommand*{\hg}{\pi} 
\newcommand*{\comp}{\circ} 
\newcommand*{\Int}{\mathbb{Z}} 
\newcommand*{\Rat}{\mathbb{Q}} 
\newcommand*{\FF}{\mathbb{F}} 
\newcommand*{\cmplx}[1]{#1^\bullet} 
\newcommand*{\tensor}{\otimes} 
\newcommand*{\isomorph}{\cong} 
\newcommand*{\op}{{o\!p}} 
\newcommand*{\cont}{{c\!\!\:o\!n\!t}} 
\newcommand*{\algc}[1]{\overline{#1}} 
\newcommand*{\sheaf}[1]{\mathcal{#1}} 
\newcommand*{\cimm}
{\mathrel{\text{$\hookrightarrow\mkern-12mu\shortmid\mkern+8mu$}}} 
\newcommand*{\oimm}
{\mathrel{\text{$\hookrightarrow\mkern-20mu\circ\mkern+6mu$}}} 
\newcommand{\openideals}{\mathfrak{I}} 
\newcommand{\Sect}{\Gamma} 
\newcommand{\Sectc}{\Gamma\!_c} 
\newcommand{\Kspace}{\mathbb{K}} 

\newcommand{\ringtransf}{\Psi} 
\newcommand{\id}{\mathrm{id}} 
\newcommand{\Frob}{\mathfrak{F}} 

\DeclareMathOperator{\co}{co} 
\DeclareMathOperator{\w}{w} 
\DeclareMathOperator{\HF}{H} 
\DeclareMathOperator{\Jac}{Jac} 
\DeclareMathOperator{\End}{End} 
\DeclareMathOperator{\im}{im} 
\DeclareMathOperator{\Spec}{Spec} 
\DeclareMathOperator{\God}{G} 
\DeclareMathOperator{\RDer}{R} 
\DeclareMathOperator{\Gal}{Gal} 
\DeclareMathOperator{\KTh}{K} 
\DeclareMathOperator{\SKTh}{SK} 

\newdir{ >}{{}*!/-5pt/@{>}} 
\newdir{O}{{}*-\cir<2.5pt>{}} 
\xyoption{2cell} 
\UseTwocells %
\xyoption{poly}

\allowdisplaybreaks[1]

\begin{document}


\title[Noncommutative $L$-functions for varieties over finite fields]{Noncommutative $L$-functions for varieties over finite fields}%
\author{Malte Witte}%

\address{Malte Witte\newline Fakult\"at f\"ur Mathematik\newline Universit\"at Regensburg\newline Universit\"atsstra{\ss}e 31\newline D-93053 Regensburg\newline Germany}%
\email{Malte.Witte@mathematik.uni-regensburg.de}%


\date{\today}%

\begin{abstract}
In this article we prove a Grothendieck trace formula for
$L$-functions of not necessarily commutative adic sheaves.
\end{abstract}

\maketitle
\section{Introduction}

Let $\sheaf{F}$ be an $\ell$-adic sheaf on a separated scheme $X$
over a finite field $\FF$ of characteristic different from $\ell$.
The $L$-function of $\sheaf{F}$ is defined as the product over all
closed points $x$ of $X$ of the characteristic polynomials of the
geometric Frobenius automorphism $\Frob_x$ at $x$ acting on the
stalk $\sheaf{F}_x$:
$$
L(X,\sheaf{F},T)=\prod_{x}\det(1-\Frob_x T^{\deg x}\colon
\sheaf{F}_x)^{-1}.
$$

The Grothendieck trace formula relates the $L$-function to the
action of the geometric Frobenius $\Frob_{\FF}$ on the $\ell$-adic
cohomology groups with proper support over the base change
$\algc{X}$ of $X$ to the algebraic closure:
$$
L(X,\sheaf{F},T)=\prod_{i\in\Int}\det(1-\Frob_{\FF} T\colon
\HF_c^i(\algc{X},\sheaf{F}))^{(-1)^{i+1}}.
$$
It was used by Grothendieck to establish the rationality and the
functional equation of the zeta function of $X$, both of which are
parts of the Weil conjectures.

The Grothendieck trace formula may also be viewed as an equality
between two elements of the first $\KTh$-group of the power series
ring $\Int_{\ell}[[T]]$. Since the ring $\Int_{\ell}[[T]]$ is a
semilocal commutative ring, $\KTh_1(\Int_{\ell}[[T]])$ may be
identified with the group of units $\Int_{\ell}[[T]]^{\times}$ via
the map induced by the determinant. For each closed point $x$ of
$X$, the $\Int_{\ell}[[T]]$-automorphism $1-\Frob_xT$ on
$\Int_{\ell}[[T]]\tensor_{\Int_{\ell}}\sheaf{F}_x$ defines a class
in $\KTh_1(\Int_{\ell}[[T]])$. The product of all these classes
converges in the profinite topology induced on
$\KTh_1(\Int_{\ell}[[T]])$ by the isomorphism
$$
\KTh_1(\Int_{\ell}[[T]])\isomorph\varprojlim_{n}\KTh_1(\Int_{\ell}[[T]]/(\ell^n,T^n)).
$$
The image of the limit under the determinant map agrees with the
inverse of the $L$-function of $\sheaf{F}$. On the other hand, the
$\Int_{\ell}[[T]]$-automorphisms
$$
\Int_{\ell}[[T]]\tensor_{\Int_{\ell}}\HF_c^i(\algc{X},\sheaf{F})\xrightarrow{1-\Frob_{\FF}
T}
\Int_{\ell}[[T]]\tensor_{\Int_{\ell}}\HF_c^i(\algc{X},\sheaf{F})
$$
also give rise to elements in the group
$\KTh_1(\Int_{\ell}[[T]])$. The Grothendieck trace formula may
thus be translated into an equality between the alternating
product of those elements and the class corresponding to the
$L$-function.

In this article, we will show that in the above reformulation of
the Grothendieck trace formula, one may replace $\Int_{\ell}$ by
any adic $\Int_{\ell}$-algebra, i.\,e.\ a compact, semilocal
$\Int_{\ell}$-algebra $\Lambda$ whose Jacobson radical is finitely
generated. These rings play an important role in noncommutative
Iwasawa theory.

A central step in this reformation is the development of a
convenient framework, in which one can put the $\KTh$-theoretic
machinery to use. This was accomplished in \cite{Witte:PhD}, using
the notion of Waldhausen categories. For any adic ring $\Lambda$,
we introduced in \emph{loc.\,cit.\ }a Waldhausen category of
perfect complexes of adic sheaves of $\Lambda$-modules on $X$.
Furthermore, we presented an explicit construction of a Waldhausen
exact functor $\RDer\Sectc(\algc{X},\cmplx{\sheaf{F}})$ that
computes the cohomology with proper support for any perfect
complex $\cmplx{\sheaf{F}}$.

By suitably adapting the classical construction, we define the
$L$-function of such a complex $\cmplx{\sheaf{F}}$ as an element
$L(\cmplx{\sheaf{F}},T)$ of $\KTh_1(\Lambda[[T]])$. The
automorphism $1-\Frob_{\FF} T$ on
$\Lambda[[T]]\tensor_{\Lambda}\RDer\Sectc(\algc{X},\cmplx{\sheaf{F}})$
gives rise to another class in $\KTh_1(\Lambda[[T]])$. Below, we
shall prove the following theorem.

\begin{thm}
Let $\cmplx{\sheaf{F}}$ be a perfect complex of adic sheaves of
$\Lambda$-modules on $X$. Then
$$
L(\cmplx{\sheaf{F}},T)=
\left[\Lambda[[T]]\tensor_{\Lambda}\RDer\Sectc(\algc{X},\cmplx{\sheaf{F}})\xrightarrow{1-\Frob_{\FF}
T}\Lambda[[T]]\tensor_{\Lambda}\RDer\Sectc(\algc{X},\cmplx{\sheaf{F}})\right]^{-1}
$$
in $\KTh_1(\Lambda[[T]])$.
\end{thm}

The rough line of argumentation in the proof is as follows. As in
the proof of the classical Grothendieck trace formula, one may
reduce everything to the case of $X$ being a smooth geometrically
connected curve over the finite field $\FF$. Moreover, one can
replace $\Lambda$ by $\Int_{\ell}[\Gal(L/K)]$, where $L$ is a
Galois extension of the function field $K$ of $X$. By the
classical Grothendieck trace formula, we know that our theorem is
true if we further enlarge $\Int_{\ell}[\Gal(L/K)]$ to the maximal
order $M$ in a split semisimple algebra. The crucial step is then
to show that the kernel of
$$
\KTh_1(\Int_{\ell}[\Gal(L/K)][[T]])\rightarrow \KTh_1(M[[T]])
$$
vanishes in the limit as $L$ tends to the separable closure of
$K$. This is achieved by using results of
\cite{Oliver:WhiteheadGroups} and \cite{FK:CNCIT}.

\emph{Outline.} Sec\-tion~\ref{sec:Waldhausen} recalls briefly
Waldhausen's construction of algebraic $\KTh$-theory. In
Sec\-tion~\ref{sec:adic rings} we introduce a special Waldhausen
category that computes the $\KTh$-theory of an adic ring. A
similar construction is then used in Sec\-tion~\ref{sec:perf
complexes} to define the categories of perfect complexes of adic
sheaves. In Sec\-tion~\ref{sec:on KTh_1} we study the first
$\KTh$-group of $\Int_{\ell}[G][[T]]$ and prove the abovementioned
vanishing result. In Sec\-tion~\ref{sec:L-functions} we define the
$L$-function of a perfect complex of adic sheaves.
Sec\-tion~\ref{sec:Grothendieck trace formula} contains the proof
of the Grothendieck trace formula for these $L$-functions.

\emph{Acknowledgements.} The author would like to thank Annette
Huber and Alexander Schmidt for their encouragement and for
valuable discussions.

\section{Waldhausen Categories}\label{sec:Waldhausen}

Waldhausen \cite{Wal:AlgKTheo} introduced a construction of
algebraic $\KTh$-theory that is both more transparent and more
flexible than Quillen's original approach. He associates
$\KTh$-groups to any category of the following kind.

\begin{defn}\label{defn:Waldhausen cat}
A \emph{Waldhausen category} $\cat W$ is a category with a zero
object~$*$, together with two subcategories $\co(\cat W)$
(\emph{cofibrations}) and $\w(\cat W)$ (\emph{weak equivalences})
subject to the following set of axioms.
\begin{enumerate}
    \item\label{enum:Waldhausen cat.isos} Any isomorphism in $\cat W$ is a morphism in $\co(\cat
    W)$ and $\w(\cat W)$.
    \item\label{enum:Waldhausen cat.0-maps} For every object $A$ in $\cat W$, the unique map $*\mto
    A$ is in $\co(\cat W)$.
    \item\label{enum:Waldhausen cat.pushouts} If $A\mto B$ is a map in $\co(\cat W)$ and $A\mto C$ is a
    map in $\cat W$, then the pushout $B\cup_A C$ exists and the
    canonical map $C\mto B\cup_A C$ is in $\co(\cat W)$.
    \item\label{enum:Waldhausen cat.glueing} If in the commutative diagram
   $$
   \xymatrix{
   B\ar[d]&A\ar[r]_f\ar[l]\ar[d]&C\ar[d]\\
   B'&A'\ar[r]_g\ar[l]&C'
   }
   $$
   the morphisms $f$ and $g$ are cofibrations and the
   downwards pointing arrows are weak equivalences, then the natural
   map $B\cup_A C\mto B'\cup_{A'}C'$ is a weak equivalence.
\end{enumerate}
\end{defn}

We denote maps from $A$ to $B$ in $\co(\cat W)$ by $A\cto B$,
those in $\w(\cat W)$ by $A\wto B$. If $C=B\cup_A *$ is a cokernel
of the cofibration $A\cto B$, then we denote the natural quotient
map from $B$ to $C$ by $B\qto C$. The sequence
$$
A\cto B\qto C
$$
is called \emph{exact sequence} or \emph{cofibre sequence}.

\begin{defn}
A functor between Waldhausen categories is called
\emph{(Waldhausen) exact} if it preserves cofibrations, weak
equivalences, and pushouts along cofibrations.
\end{defn}

If $\cat{W}$ is a Waldhausen category, then Waldhausen's
$S$-construction yields a topological space $\Kspace(\cat{W})$ and
Waldhausen exact functors $F\colon \cat{W}\mto \cat{W'}$ yield
continuous maps $\Kspace(F)\colon \Kspace(\cat{W})\mto
\Kspace(\cat{W'})$ \cite{Wal:AlgKTheo}.

\begin{defn}
The \emph{$n$-th $\KTh$-group of $\cat{W}$} is defined to be the
$n$-th homotopy group of $\Kspace(\cat{W})$:
$$
\KTh_n(\cat{W})=\hg_n(\Kspace(\cat{W})).
$$
\end{defn}

\begin{exmpl}\label{exmpl:Waldhausen categories}
\quad
\begin{enumerate}
\item Any exact category $\cat{E}$ may be viewed as a Waldhausen
category by taking the admissible monomorphisms as cofibrations
and isomorphisms as weak equivalences. Then the Waldhausen
$\KTh$-groups of $\cat{E}$ agree with the Quillen $\KTh$-groups of
$\cat{E}$ \cite[Theorem 1.11.2]{ThTr:HAKTS+DC}.

\item Let $\cat{Kom}^b(\cat{E})$ be the category of bounded
complexes over the exact category $\cat{E}$ with degreewise
admissible monomorphisms as cofibrations and quasi-isomorphisms
(in the category of complexes of an ambient abelian category
$\cat{A}$) as weak equivalences. By the Gillet-Waldhausen theorem
\cite[Theorem 1.11.7]{ThTr:HAKTS+DC}, the Waldhausen $\KTh$-groups
of $\cat{Kom}^b(\cat{E})$ also agree with the $\KTh$-groups of
$\cat{E}$ .

\item\label{enum:Derived equivalences} In fact, Thomason showed
that if $\cat{W}$ is any sufficiently nice Waldhausen category of
complexes and $F\colon \cat{W}\mto \cat{Kom}^b(\cat{E})$ a
Waldhausen exact functor that induces an equivalence of the
derived categories of $\cat{W}$ and $\cat{Kom}^b(\cat{E})$, then
$F$ induces an isomorphism of the corresponding $\KTh$-groups
\cite[Theorem 1.9.8]{ThTr:HAKTS+DC}.
\end{enumerate}
\end{exmpl}

\begin{rem}
In the view of Example~\ref{exmpl:Waldhausen
categories}.(\ref{enum:Derived equivalences}) one might wonder
wether it is possible to define a reasonable $\KTh$-theory for
triangulated categories. However, \cite{Schlichting:NoteOnKTheory}
shows that such a construction fails to exist.
\end{rem}

The zeroth $\KTh$-group of a Waldhausen category can be described
fairly explicitly as follows.

\begin{prop}
Let $\cat{W}$ be a Waldhausen category. The group
$\KTh_0(\cat{W})$ is the abelian group generated by the objects of
$\cat{W}$ modulo the relations
\begin{enumerate}
\item $[A]=[B]$ if there exists a weak equivalence $A\wto B$,

\item $[B]=[A][C]$ if there exists a cofibre sequence $A\cto B\qto
C$.
\end{enumerate}
\end{prop}
\begin{proof}
See \cite[\S 1.5.6]{ThTr:HAKTS+DC}.
\end{proof}

There also exists a description of $\KTh_1(\cat{W})$ for general
$\cat{W}$ as the kernel of a certain group homomorphism
\cite{MT:1TWKTS}. We shall come back to this description later in
a more specific situation.

\section{The $\KTh$-Theory of Adic Rings}\label{sec:adic rings}

All rings will be associative with unity, but not necessarily
commutative. For any ring $R$, we let
$$
\Jac(R)=\set{x\in R|\text{$1-rx$ is invertible for any $r\in R$}}
$$
denote the \emph{Jacobson radical} of $R$, i.\,e.\ the
intersection of all maximal left ideals. It is the largest
two-sided ideal $I$ of $R$ such that $1+I\subset R^{\times}$
\cite[Chapter~2, $\S4$]{Lam:FirstCourseNoncomRings}. The ring $R$
is called \emph{semilocal} if $R/\Jac(R)$ is artinian.

\begin{defn}
A ring $\Lambda$ is called an \emph{adic ring} if it satisfies any
of the following equivalent conditions:
\begin{enumerate}
\item $\Lambda$ is compact, semilocal and the Jacobson radical is
finitely generated.

\item For each integer $n\geq 1$, the ideal $\Jac(\Lambda)^n$ is
of finite index in $\Lambda$ and
$$
\Lambda=\varprojlim_{n}\Lambda/\Jac(\Lambda)^n.
$$
\item There exists a twosided ideal $I$ such that for each integer
$n\geq 1$, the ideal $I^n$ is of finite index in $\Lambda$ and
$$
\Lambda=\varprojlim_{n}\Lambda/I^n.
$$
\end{enumerate}
\end{defn}

\begin{exmpl}
The following rings are adic rings:
\begin{enumerate}
\item any finite ring,

\item $\Int_{\ell}$,

\item the group ring $\Lambda[G]$ for any finite group $G$ and any
adic ring $\Lambda$,

\item the power series ring $\Lambda[[T]]$ for any adic ring
$\Lambda$ and an indeterminate $T$ that commutes with all elements
of $\Lambda$,

\item the profinite group ring $\Lambda[[G]]$, when $\Lambda$ is a
adic $\Int_{\ell}$-algbra and $G$ is a profinite group whose
$\ell$-Sylow subgroup has finite index in $G$.
\end{enumerate}
Note that adic rings are not noetherian in general, the power
series over $\Int_{\ell}$ in two noncommuting indeterminates being
a counterexample.
\end{exmpl}

We will now examine the $\KTh$-theory of $\Lambda$.

\begin{defn}
Let $R$ be any ring. A complex $\cmplx{M}$ of left $R$-modules is
called \emph{strictly perfect} if it is strictly bounded and for
every $n$, the module $M^n$ is finitely generated and projective.
We let $\cat{SP}(R)$ denote the Waldhausen category of strictly
perfect complexes, with quasi-isomorphisms as weak equivalences
and injective complex morphisms as cofibrations.
\end{defn}

\begin{defn}\label{defn:complexes of bimodules}
Let $R$ and $S$ be two rings. We denote by
$R^{\op}\text{-}\cat{SP}(S)$ the Waldhausen category of complexes
of $S$-$R$-bimodules (with $S$ acting from the left, $R$ acting
from the right) which are strictly perfect as complexes of
$S$-modules. The weak equivalences are given by
quasi-isomorphisms, the cofibrations are the injective complex
morphisms.
\end{defn}

By Example~\ref{exmpl:Waldhausen categories} we know that the
Waldhausen $\KTh$-theory of $\cat{SP}(R)$ coincides with the
Quillen $\KTh$-theory of $R$:
$$\KTh_n(\cat{SP}(R))=\KTh_n(R).$$
For complexes $\cmplx{M}$ and $\cmplx{N}$ of right and left
$R$-modules, respectively, we let
$$
\cmplx{(M\tensor_R N)}
$$
denote
the total complex of the bicomplex $\cmplx{M}\tensor_R \cmplx{N}$.
Any complex $\cmplx{M}$ in $R^{\op}\text{-}\cat{SP}(S)$ clearly
gives rise to a Waldhausen exact functor
$$
\cmplx{(M\tensor_R(-))}\colon \cat{SP}(R)\mto \cat{SP}(S).
$$
and hence, to homomorphisms $ \KTh_n(R)\mto \KTh_n (S)$.

Let now $\Lambda$ be an adic ring. The first algebraic
$\KTh$-group of $\Lambda$ has the following useful property.

\begin{prop}[\cite{FK:CNCIT}, Prop. 1.5.3]\label{prop:K_1 of adic rings}
Let $\Lambda$ be an adic ring. Then
$$
\KTh_1(\Lambda)=\varprojlim_{I\in\openideals_{\Lambda}}\KTh_1(\Lambda/I)
$$
In particular, $\KTh_1(\Lambda)$ is a profinite group.
\end{prop}

It will be convenient to introduce another Waldhausen category
that computes the $\KTh$-theory of $\Lambda$.

\begin{defn}
Let $R$ be any ring. A complex $\cmplx{M}$ of left $R$-modules is
called \emph{$DG$-flat} if every module $M^n$ is flat and for
every acyclic complex $\cmplx{N}$ of right $R$-modules, the
complex $\cmplx{(N\tensor_R M)}$ is acyclic.
\end{defn}

We shall denote the lattice of open ideals of an adic ring
$\Lambda$ by $\openideals_{\Lambda}$.

\begin{defn}\label{defn:PDG(Lambda)}
Let $\Lambda$ be an adic ring. We denote by
$\cat{PDG}^{\cont}(\Lambda)$ the following Waldhausen category.
The objects of $\cat{PDG}^{\cont}(\Lambda)$ are inverse system
$(\cmplx{P}_I)_{I\in \openideals_{\Lambda}}$ satifying the
following conditions:
\begin{enumerate}
\item for each $I\in\openideals_{\Lambda}$, $\cmplx{P}_I$ is a
$DG$-flat complex of left $\Lambda/I$-modules and \emph{perfect},
i.\,e.\ quasi-isomorphic to a complex in $\cat{SP}(\Lambda)$,

\item for each $I\subset J\in\openideals_{\Lambda}$, the
transition morphism of the system
$$
\varphi_{IJ}:\cmplx{P}_I\mto \cmplx{P}_J
$$
induces an isomorphism
$$
\Lambda/J\tensor_{\Lambda/I}\cmplx{P}_I\isomorph \cmplx{P}_J.
$$
\end{enumerate}
A morphism of inverse systems $(f_I\colon \cmplx{P}_I\mto
\cmplx{Q}_I)_{I\in\openideals_{\Lambda}}$ in
$\cat{PDG}^{\cont}(\Lambda)$ is a weak equivalence if every $f_I$
is a quasi-isomorphism. It is a cofibration if every $f_I$ is
injective.
\end{defn}

The following proposition is an easy consequence of Waldhausen's
approximation theorem.

\begin{prop}\label{prop:embedding SP(Lambda) in PDG(Lambda)}
The Waldhausen exact functor
$$
\cat{SP}(\Lambda)\mto\cat{PDG}^{\cont}(\Lambda),\qquad
\cmplx{P}\mto
(\Lambda/I\tensor_{\Lambda}\cmplx{P})_{I\in\openideals_{\Lambda}}
$$
identifies $\cat{SP}(\Lambda)$ with a full Waldhausen subcategory
of $\cat{PDG}^{\cont}(\Lambda)$. Moreover, it induces isomorphisms
$$
\KTh_n(\cat{SP}(\Lambda))\isomorph\KTh_n(\cat{PDG}^{\cont}(\Lambda)).
$$
\end{prop}
\begin{proof}
See \cite[Proposition~5.2.5]{Witte:PhD}.
\end{proof}

We will now extend the definition of the tensor product to
$\cat{PDG}^{\cont}(\Lambda)$.

\begin{defn}\label{defn:change of ring functor}
For $(\cmplx{P}_I)_{I\in\openideals_{\Lambda}}\in
\cat{PDG}^{\cont}(\Lambda)$ and $\cmplx{M}\in
\Lambda^{\op}\text{-}\cat{SP}(\Lambda')$ we set
$$
\ringtransf_{M}\left((\cmplx{P}_I)_{I\in\openideals_{\Lambda}}\right)=
(\varprojlim_{J\in\openideals_{\Lambda}}
\Lambda'/I\tensor_{\Lambda'}\cmplx{(M\tensor_{\Lambda}P_{J})})_{I\in\openideals_{\Lambda'}}
$$
and obtain a Waldhausen exact functor
$$
\ringtransf_{M}\colon
\cat{PDG}^{\cont}(\Lambda)\mto\cat{PDG}^{\cont}(\Lambda').
$$
\end{defn}

\begin{prop}
Let $\cmplx{M}$ be a complex in
$\Lambda^{\op}\text{-}\cat{SP}(\Lambda')$. Then the following
diagram commutes.
$$
\xymatrix{
\KTh_n(\cat{SP}(\Lambda))\ar[r]^{\isomorph}\ar[d]^{\KTh_n(\cmplx{M}\tensor_{\Lambda}(-))}&
\KTh_n(\cat{PDG}^{\cont}(\Lambda))\ar[d]^{\KTh_n(\ringtransf_{M})}\\
\KTh_n(\cat{SP}(\Lambda'))\ar[r]^{\isomorph}&
\KTh_n(\cat{PDG}^{\cont}(\Lambda')) }
$$
\end{prop}
\begin{proof}
Let $\cmplx{P}$ be a strictly perfect complex in
$\cat{SP}(\Lambda)$. There exists a canonical isomorphism
$$
\left(\Lambda'/I\tensor_{\Lambda'}\cmplx{(M\tensor_{\Lambda}
P)}\right)_{I\in\openideals_{\Lambda'}}\isomorph
(\varprojlim_{J\in\openideals_{\Lambda}}
\Lambda'/I\tensor_{\Lambda'}\cmplx{(M\tensor_{\Lambda}\Lambda/J\tensor_{\Lambda}P)})_{I\in\openideals_{\Lambda'}}.
$$
\end{proof}

From \cite{MT:1TWKTS} we deduce the following properties of the
group $\KTh_1(\Lambda)$.

\begin{prop}\label{prop:presentation of K_1}
The group $\KTh_1(\Lambda)$ is  generated by the weak
autoequivalences $(f_I\colon \cmplx{P}_I\wto
\cmplx{P}_I)_{I\in\openideals_{\Lambda}}$ in
$\cat{PDG}^{\cont}(\Lambda)$. Moreover, we have the following
relations:
\begin{enumerate}
\item $[(f_I\colon\cmplx{P}_I\wto
\cmplx{P}_I)_{I\in\openideals_{\Lambda}}]=[(g_I\colon\cmplx{P}_I\wto
\cmplx{P}_I)_{I\in\openideals_{\Lambda}}][(h_I\colon\cmplx{P}_I\wto
\cmplx{P}_I)_{I\in\openideals_{\Lambda}}]$ if for each
$I\in\openideals_{\Lambda}$, one has $f_I=g_I\comp h_I$,

\item $[(f_I\colon\cmplx{P}_I\wto
\cmplx{P}_I)_{I\in\openideals_{\Lambda}}]=[(g_I\colon\cmplx{Q}_I\wto
\cmplx{Q}_I)_{I\in\openideals_{\Lambda}}]$ if for each
$I\in\openideals_{\Lambda}$, there exists a quasi-isomorphism
$a_I\colon \cmplx{P}_I\wto \cmplx{Q}_I$ such that the square
$$
\xymatrix{\cmplx{P}_I\ar[r]^{f_I}\ar[d]^{a_I}&\cmplx{P}_I\ar[d]^{a_I}\\
\cmplx{Q}_I\ar[r]^{g_I}&\cmplx{Q}_I}
$$
commutes up to homotopy,

\item $[(g_I\colon\cmplx{P'}_I\wto
\cmplx{P'}_I)_{I\in\openideals_{\Lambda}}]=[(f_I\colon\cmplx{P}_I\wto
\cmplx{P}_I)_{I\in\openideals_{\Lambda}}][(h_I\colon\cmplx{P''}_I\wto
\cmplx{P''}_I)_{I\in\openideals_{\Lambda}}]$ if for each
$I\in\openideals_{\Lambda}$, there exists an exact sequence $P\cto
P'\qto P''$ such that the diagram
$$
\xymatrix{\cmplx{P}_I\ar[d]^{f_I}\ar@{>->}[r]&\cmplx{P'}_I\ar[d]^{g_I}\ar@{>>}[r]&\cmplx{P''}\ar[d]^{h_I}\\
\cmplx{P}_I\ar@{>->}[r]&\cmplx{P'}_I\ar@{>>}[r]&\cmplx{P''}}
$$
commutes in the strict sense.
\end{enumerate}
\end{prop}
\begin{proof}
The description of $\KTh_1(\cat{PDG}^{\cont}(\Lambda))$ as the
kernel of
$$
\mathcal{D}_1\cat{PDG}^{\cont}(\Lambda)\xrightarrow{\del}\mathcal{D}_0\cat{PDG}^{\cont}(\Lambda)
$$
given in \cite{MT:1TWKTS} shows that the weak autoequivalences are
indeed elements of $\KTh_1(\cat{PDG}^{\cont}(\Lambda))$. Together
with Pro\-po\-si\-tion~\ref{prop:K_1 of adic rings} this
description also implies that relations (1) and (3) are satisfied.
For relation (2), one can use \cite[Lemma 3.1.6]{Witte:PhD}.
Finally, the classical description of $\KTh_1(\Lambda)$ implies
that $\KTh_1(\cat{PDG}^{\cont}(\Lambda))$ is already generated by
isomorphisms of finitely generated, projective modules viewed as
strictly perfect complexes concentrated in degree $0$.
\end{proof}

\begin{rem}
Despite the relatively explicit description of $\KTh_1(\cat{W})$
for a Waldhausen category $\cat{W}$ in \cite{MT:1TWKTS} it is not
an easy task to deduce from it a presentation of $\KTh_1(\cat{W})$
as an abelian group. We refer to \cite{MT:OnK1WaldCat} for a
partial result in this direction.

In particular, one should not expect that the relations (1)--(3)
describe the group $\KTh_1(\cat{PDG}^{\cont}(\Lambda))$
completely. However, they will suffice for the purpose of this
paper.
\end{rem}

\section{Perfect Complexes of Adic Sheaves}\label{sec:perf
complexes}

We let $\FF$ denote a finite field of characteristic $p$, with
$q=p^{\nu}$ elements. Furthermore, we fix an algebraic closure
$\algc{\FF}$ of $\FF$.

For any scheme $X$ in the category $\cat{Sch}_{\FF}^{sep}$ of
separated $\FF$-schemes of finite type and any adic ring $\Lambda$
we introduced in \cite{Witte:PhD} a Waldhausen category
$\cat{PDG}^{\cont}(X,\Lambda)$ of perfect complexes of adic
sheaves on $X$. Below, we will recall the definition.

\begin{defn}
Let $R$ be a finite ring and $X$ be a scheme in
$\cat{Sch}_{\FF}^{sep}$. A complex $\cmplx{\sheaf{F}}$ of \'etale
sheaves of left $R$-modules on $X$ is called \emph{strictly
perfect} if it is strictly bounded and each $\sheaf{F}^n$ is
constructible and flat. A complex is called \emph{perfect} if it
is quasi-isomorphic to a strictly perfect complex. It is
\emph{$DG$-flat} if for each geometric point of $X$, the complex
of stalks is $DG$-flat.
\end{defn}

\begin{defn}\label{defn:PDG(X,R)}
We will denote by $\cat{PDG}(X,R)$ the \emph{category of $DG$-flat
perfect complexes of $R$-modules} on $X$. It is a Waldhausen
category with quasi-isomorphisms as weak equivalences and
injective complex morphisms as cofibrations.
\end{defn}

\begin{defn}\label{defn:PDGcont(X,Lambda)}
Let $X$ be a scheme in $\cat{Sch}_{\FF}$ and let $\Lambda$ be an
adic ring. The \emph{category of perfect complexes of adic
sheaves} $\cat{PDG}^{\cont}(X,\Lambda)$ is the following
Waldhausen category. The objects of $\cat{PDG}^{\cont}(X,\Lambda)$
are inverse system $(\cmplx{\sheaf{F}}_I)_{I\in
\openideals_{\Lambda}}$ such that:
\begin{enumerate}
\item for each $I\in\openideals_{\Lambda}$, $\cmplx{\sheaf{F}}_I$
is in $\cat{PDG}(X,\Lambda/I)$,

\item for each $I\subset J\in\openideals_{\Lambda}$, the
transition morphism
$$
\varphi_{IJ}:\cmplx{\sheaf{F}}_I\mto \cmplx{\sheaf{F}}_J
$$
of the system induces an isomorphism
$$
\Lambda/J\tensor_{\Lambda/I}\cmplx{\sheaf{F}}_I\wto
\cmplx{\sheaf{F}}_J.
$$
\end{enumerate}
Weak equivalences and cofibrations are those morphisms of inverse
systems that are weak equivalences or cofibrations for each
$I\in\openideals_{\Lambda}$, respectively.
\end{defn}

\begin{rem}\label{rem:PDGcont=PDG}
If $\Lambda$ is a finite ring, the zero ideal is open and hence,
an element in $\openideals_{\Lambda}$. In particular, the
following Waldhausen exact functors are mutually inverse
equivalences for finite rings $\Lambda$:
\begin{align*}
\cat{PDG}^{\cont}(X,\Lambda)\mto \cat{PDG}(X,\Lambda),&\qquad
(\cmplx{\sheaf{F}}_I)_{I\in\openideals_{\Lambda}}\mapsto
\cmplx{\sheaf{F}}_{(0)},\\
\cat{PDG}(X,\Lambda)\mto \cat{PDG}^{\cont}(X,\Lambda),&\qquad
\cmplx{\sheaf{F}}\mapsto
(\Lambda/I\tensor_{\Lambda}\cmplx{\sheaf{F}})_{I\in\openideals_{\Lambda}}.
\end{align*}
We use these functors to identify the two categories.
\end{rem}

If $\Lambda=\Int_{\ell}$, then the subcategory of complexes
concentrated in degree $0$ of $\cat{PDG}^{\cont}(X,\Int_{\ell})$
corresponds precisely to the exact category of flat constructible
$\ell$-adic sheaves on $X$ in the sense of \cite[Expos\'e~VI,
Definition~1.1.1]{SGA5}. In this sense, we recover the classical
theory.

If $f\colon Y\mto X$ is a morphism of schemes, we define a
Waldhausen exact functor
$$
f^*\colon\cat{PDG}^{\cont}(X,\Lambda)\mto\cat{PDG}^{\cont}(Y,\Lambda),\qquad
(\cmplx{\sheaf{F}}_I)_{I\in\openideals_{\Lambda}}\mapsto
(f^*\cmplx{\sheaf{F}}_I)_{I\in\openideals_{\Lambda}}.
$$
We will also need a Waldhausen exact functor that computes higher
direct images with proper support. For the purposes of this
article it suffices to use the following construction.

\begin{defn}
Let $f\colon X\mto Y$ be a morphism in $\cat{Sch}_{\FF}^{sep}$.
Then there exists a factorisation $f=p\comp j$ with $j\colon
X\oimm X'$ an open immersion and $p\colon X'\mto Y$ a proper
morphism. Let $\cmplx{\God}_{X'}\sheaf{G}$ denote the Godement
resolution of a complex $\cmplx{\sheaf{G}}$ of abelian \'etale
sheaves on $X'$. Define
\begin{align*}
\RDer f_!\colon \cat{PDG}^{\cont}(X,\Lambda)&\mto
\cat{PDG}^{\cont}(Y,\Lambda)\\
(\cmplx{\sheaf{F}}_I)_{I\in\openideals_{\Lambda}}&\mapsto(f_*\cmplx{\God}_{X'}
 j_!\sheaf{F}_I)_{I\in\openideals_{\Lambda}}
\end{align*}
\end{defn}

Obviously, this definition depends on the particular choice of the
compactification $f=p\comp j$. However, all possible choices will
induce the same homomorphisms
$$
\KTh_n(\RDer f_!)\colon \KTh_n(\cat{PDG}^{\cont}(X,\Lambda))\mto
\KTh_n(\cat{PDG}^{\cont}(Y,\Lambda))
$$
and this is all we need.

\begin{rem}
In \cite[Section~4.5]{Witte:PhD} we present a way to make the
construction of $\RDer f_!$ independent of the choice of a
particular compactification.
\end{rem}

\begin{prop}\label{prop:properties of Rf_!}
Let $f\colon X\mto Y$ be a morphism in $\cat{Sch}_{\FF}^{sep}$.
\begin{enumerate}
\item $\KTh_n(\RDer f_!)$ is independent of the choice of a
compactification $f=p\comp j$.

\item Let $\FF'$ be a subfield of $\FF$ and consider $f$ as a
morphism in $\cat{Sch}_{\FF'}^{sep}$. Then $\KTh_n(\RDer f_!)$
remains the same.

\item If $g\colon Y\mto Z$ is another morphism in
$\cat{Sch}_{\FF}^{sep}$, then
$$
\KTh_n(\RDer (g\comp f)_!)=\KTh_n(\RDer g_!)\comp \KTh_n(\RDer
f_!)
$$

\item For any cartesian square
$$
\xymatrix{Y\times_X Z\ar[r]^{f'}\ar[d]_{g'}&Z\ar[d]^g\\
 Y\ar[r]^{f}&X}
$$
in $\cat{Sch}_{\FF}^{sep}$ we have
$$
\KTh_n(f^*\RDer g_!)=\KTh_n(\RDer g'_!{f'}^*)
$$
\end{enumerate}
\end{prop}
\begin{proof}
All of this follows easily from \cite[Expos\'e~XXVII]{SGA4-3}. See
also \cite[Section 4.5.]{Witte:PhD}.
\end{proof}

\begin{defn}
Let $X$ be a scheme in $\cat{Sch}_{\FF}^{sep}$ and write $h\colon
X\mto \Spec \FF$ for the structure map, $s\colon \Spec
\algc{\FF}\mto \Spec \FF$ for the map induced by the embedding
into the algebraic closure. We define the Waldhausen exact functor
$$
\RDer_\FF\Sectc(\algc{X},-)\colon\cat{PDG}^{\cont}(X,\Lambda)\mto\cat{PDG}^{\cont}(\Lambda)
$$
to be the composition of
$$
\RDer h_!\colon \cat{PDG}^{\cont}(X,\Lambda)\mto
\cat{PDG}^{\cont}(\Spec \FF,\Lambda)
$$
with the section functor
$$
\cat{PDG}^{\cont}(\Spec
\FF,\Lambda)\mto\cat{PDG}^{\cont}(\Lambda), \qquad
(\cmplx{\sheaf{F}}_{I})_{I\in\openideals_{\Lambda}}\mto
(\Sect(\Spec
\algc{\FF},s^*\cmplx{\sheaf{F}}_I))_{I\in\openideals_{\Lambda}}.
$$
\end{defn}

\begin{rem}
If $\FF'$ is a subfield of $\FF$, then
$\RDer_\FF\Sectc(\algc{X},-)$ and $\RDer_{\FF'}\Sectc(\algc{X},-)$
are in fact quasi-isomorphic and hence, they induce the same
homomorphism of $\KTh$-groups. Nevertheless, it will be convenient
to distinguish between the two functors. We will omit the index if
the base field is clear from the context.
\end{rem}

The definition of $\ringtransf_{M}$ extends to
$\cat{PDG}^{\cont}(X,\Lambda)$.

\begin{defn}
For $(\cmplx{\sheaf{F}}_I)_{I\in\openideals_{\Lambda}}\in
\cat{PDG}^{\cont}(X,\Lambda)$ and $\cmplx{M}\in
\Lambda^{\op}\text{-}\cat{SP}(\Lambda')$ we set
$$
\ringtransf_{M}\left((\cmplx{\sheaf{F}}_I)_{I\in\openideals_{\Lambda}}\right)=
(\varprojlim_{J\in\openideals_{\Lambda}}
\Lambda'/I\tensor_{\Lambda'}\cmplx{(M\tensor_{\Lambda}\sheaf{F}_{J})})_{I\in\openideals_{\Lambda'}}
$$
and obtain a Waldhausen exact functor
$$
\ringtransf_{M}\colon
\cat{PDG}^{\cont}(X,\Lambda)\mto\cat{PDG}^{\cont}(X,\Lambda').
$$
\end{defn}

\begin{prop}\label{prop:compatibility of change of rings}
Let $\cmplx{M}$ be a complex in
$\Lambda^{\op}\text{-}\cat{SP}(\Lambda')$. Then the following
diagram commutes.
$$
\xymatrix@C+1cm{
\KTh_n(\cat{PDG}^{\cont}(X,\Lambda))\ar[r]^{\KTh_n\RDer\Sectc(X,-)}\ar[d]^{\KTh_n(\ringtransf_{M})}&
\KTh_n(\cat{PDG}^{\cont}(\Lambda))\ar[d]^{\KTh_n(\ringtransf_{M})}\\
\KTh_n(\cat{PDG}^{\cont}(X,\Lambda'))\ar[r]^{\KTh_n\RDer\Sectc(X,-)}&
\KTh_n(\cat{PDG}^{\cont}(\Lambda')) }
$$
\end{prop}
\begin{proof}
See \cite[Proposition~5.5.7]{Witte:PhD}.
\end{proof}

Finally, we need the following result. Let $X$ be a connected
scheme and $f\colon Y\mto X$ a finite \emph{Galois covering} of
$X$ with \emph{Galois group} $G$ , i.\,e. $f$ is finite \'etale
and the degree of $f$ is equal to the order of
$G=\operatorname{Aut}_X(Y)$. We set
$$
\Int_{\ell}[G]_X^\sharp=f_!f^*\Int_{\ell}.
$$
Then $\Int_{\ell}[G]_X^\sharp$ is a locally constant constructible
flat sheaf of $\Int_{\ell}[G]$-modules. In fact, it corresponds to
the continous Galois module $\Int_{\ell}[G]$ on which the
fundamental group of $X$ acts contragrediently.

\begin{lem}\label{lem:twisting lemma}
Let $X$ be a connected scheme in $\cat{Sch}^{sep}_{\FF}$. Let $R$
be a finite $\Int_{\ell}$-algebra and let $\cmplx{\sheaf{F}}$ be a
bounded complex of flat, locally constant, and constructible
sheaves in $\cat{PDG}(X,R)$. Then there exists a finite Galois
covering $Y$ of $X$ with Galois group $G$ and a complex
$\cmplx{M}$ in $\Int_{\ell}[G]^{op}\text{-}\cat{SP}(R)$ such that
$$
\cmplx{\sheaf{F}}\isomorph \ringtransf_M(\Int_{\ell}[G]_X^\sharp).
$$
\end{lem}
\begin{proof}
Choose a large enough Galois covering $f\colon Y\mto X$ such that
$f^*\cmplx{\sheaf{F}}$ is a complex of constant sheaves and set
$\cmplx{M}=\Sect(Y,f^*\cmplx{\sheaf{F}})$. This is in a natural
way a complex in $\Int_{\ell}[G]^{op}\text{-}\cat{SP}(R)$ and
$$
\cmplx{M}\tensor_{\Int_{\ell}[G]}\Int_{\ell}[G]_X^\sharp\isomorph\cmplx{\sheaf{F}}
$$
See \cite[Section~5.6]{Witte:PhD} for further details.
\end{proof}

\section{On $\KTh_1(\Int_{\ell}[G][[T]])$}\label{sec:on KTh_1}

Let $G$ be a finite group and let $T$ denote an indeterminate that
commutes with every element of $\Int_{\ell}[G]$. We need some
results on the structure of $\KTh_1(\Int_{\ell}[G][[T]])$. Recall
that there exists a finite extension $F$ of $\Rat_{\ell}$ such
that $F[G]$ is \emph{split semisimiple}:
$$
F[G]\isomorph \prod_{k=1}^r \End_F(F^{s_k})
$$
for some integers $r, s_1,\dotsc, s_r$. Write $\mathcal{O}_F$ for
the valuation ring of $F$ and let $M$ be a \emph{maximal
$\Int_{\ell}$-order} in $F[G]$, i.\,e.\ an $\Int_{\ell}$-lattice
in $F[G]$ which is a subring and which is maximal with respect to
this property. Then
$$
M\isomorph \prod_{k=1}^r \End_{\mathcal{O}_F}(\mathcal{O}_F^{s_k})
$$
according to \cite[Theorem~1.9]{Oliver:WhiteheadGroups}. In
particular, the determinant map induces an isomorphism
$$
\KTh_1(M)\isomorph \bigoplus_{k=1}^r \mathcal{O}_F^{\times}.
$$
Theorem~2.5 of \emph{loc.\,cit.\ }then implies
$$
\SKTh_1(\Int_{\ell}[G])=\ker \KTh_1(\Int_{\ell}[G])\mto
\KTh_1(\Rat_{\ell}[G])=\ker \KTh_1(\Int_{\ell}[G])\mto \KTh_1(M).
$$

Analogously, we define a subgroup in
$\KTh_1(\Int_{\ell}[G][[T]])$.
\begin{defn}
Let $G$ be a finite group and choose a finite extension $F$ of
$\Rat_{\ell}$ such that $F[G]$ is split semisimple. We set
$$
\SKTh_1(\Int_{\ell}[G][[T]])=\ker \KTh_1(\Int_{\ell}[G][[T]])\mto
\KTh_1(M[[T]])
$$
where $M$ denotes the maximal $\Int_{\ell}$-order in $F[G]$.
\end{defn}

\begin{lem}\label{lem:description of SKTh_1}
For any finite group $G$,
$$
\SKTh_1(\Int_{\ell}[G][[T]])\isomorph\varprojlim_{n}
\SKTh_1(\Int_{\ell}[G\times \Int/(\ell^n)]).
$$
\end{lem}
\begin{proof}
Let $F$ and $F'$ be splitting fields for $\Int_{\ell}[G]$ and
$\Int_{\ell}[G\times\Int/(\ell^n)]$, respectively and denote the
corresponding maximal orders by $M$ and $M'$. The commutativity of
the diagram
$$
\xymatrix{
\KTh_1(M[\Int/(\ell^n)])\ar[r]\ar[d]_{\isomorph}^{\det}&\KTh_1(M')\ar[d]_{\isomorph}^{\det}\\
\bigoplus\limits_{k=1}^r
\mathcal{O}_F[\Int/(\ell^n)]^{\times}\ar[r]^{\subset}&\bigoplus\limits_{k=1}^{r'}\mathcal{O}_{F'}^{\times}
}
$$
implies that
$$
\SKTh_1(\Int_{\ell}[G\times \Int/(\ell^n)])=\ker
\KTh_1(\Int_{\ell}[G\times \Int/(\ell^n)])\mto
\KTh_1(M[\Int/(\ell^n)]).
$$
By \cite[Theorem~5.3.5]{NSW:CohomNumFields} the choice of a
topological generator $\gamma\in \Int_{\ell}$ induces an
isomorphism
$$
\Int_{\ell}[[T]]\isomorph\varprojlim_{n}\Int_{\ell}[\Int/(\ell^n)],\qquad
T\mapsto \gamma-1.
$$
In particular, we have compatible isomorphisms
$$
\KTh_1(\Int_{\ell}[G][[T]])\isomorph \varprojlim_n
\KTh_1(\Int_{\ell}[G\times \Int/(\ell^n)]),\qquad
\KTh_1(M[[T]])\isomorph \varprojlim_n \KTh_1(M[\Int/(\ell^n)]).
$$
by Pro\-po\-si\-tion~\ref{prop:K_1 of adic rings}. Hence, we
obtain an isomorphism
$$
\SKTh_1(\Int_{\ell}[G][[T]])\isomorph\varprojlim_{n}
\SKTh_1(\Int_{\ell}[G\times \Int/(\ell^n)]),
$$
as claimed.
\end{proof}

\begin{prop}\label{prop:SKTh_1}
For any finite group $G$,
$$
\SKTh_1(\Int_{\ell}[G][[T]])\isomorph \SKTh_1(\Int_{\ell}[G]).
$$
\end{prop}
\begin{proof}
By Lem\-ma~\ref{lem:description of SKTh_1} it suffices to prove
that the projection
$\Int_{\ell}[G\times\Int/(\ell^n)]\mto\Int_{\ell}[G]$ induces an
isomorphism
$$
\SKTh_1(\Int_{\ell}[G\times \Int/(\ell^n)])\isomorph
\SKTh_1(\Int_{\ell}[G]).
$$

Let $g_1,\dotsc,g_k$ be a system of representatives for the
$\Rat_{\ell}$-conjugacy classes of elements of order prime to
$\ell$ in $G$.  (Two elements $g$, $h$ of order $r$ prime to
$\ell$ are called $\Rat_{\ell}$-conjugated if $g^a=xhx^{-1}$ for
some $x\in G$, $a\in \Gal(\Rat_{\ell}(\zeta_r)/\Rat_{\ell})\subset
(\Int/(r))^{\times}$.) Let $r_i$ denote the order of $g_i$ and set
\begin{align*}
N_i(G)&=\set{x\in G\quad\mid\text{$xg_ix^{-1}=g_i^a$ for some $a\in\Gal(\Rat_{\ell}(\zeta_{r_i})/\Rat_{\ell})$}},\\
Z_i(G)&=\set{x\in G\quad\mid\text{$xg_ix^{-1}=g_i$}}.
\end{align*}
Furthermore, let
$$
\HF_2^{ab}(Z_i(G),\Int)=\im \bigoplus_{\substack{H\subset
Z_i(G)\\\text{abelian}}} \HF_2(H,\Int)\mto \HF_2(Z_i(G),\Int)
$$
denote the subgroup of the second homology group generated by
elements induced up from abelian subgroups of $Z_i(G)$. According
to \cite[Theorem~12.5]{Oliver:WhiteheadGroups} there exists an
isomorphism
$$
\SKTh_1(\Int_{\ell}[G])\isomorph \bigoplus_{i=0}^k
\HF_0(N_i(G)/Z_i(G),
\HF_2(Z_i(G),\Int)/\HF_2^{ab}(Z_i(G),\Int))_{(\ell)}.
$$
Now, $(g_1,0),\dotsc,(g_k,0)$ is a system of representatives for
the $\Rat_{\ell}$-conjugacy classes of elements of order prime to
$\ell$ in $G\times \Int/(\ell^n)$ and
$$
N_i(G\times \Int/(\ell^n))=N_i(G)\times\Int/(\ell^n),\qquad
Z_i(G\times \Int/(\ell^n))=Z_i(G)\times\Int/(\ell^n).
$$
By \emph{loc.\,cit.}, Proposition~8.12, we have
\begin{multline*}
\HF_2(Z_i(G)\times\Int/(\ell^n),\Int)/\HF_2^{ab}(Z_i(G)\times\Int/(\ell)^n,\Int)=\\
\HF_2(Z_i(G),\Int)/\HF_2^{ab}(Z_i(G),\Int)\times\HF_2(\Int/(\ell^n),\Int)/\HF_2^{ab}(\Int/(\ell^n),\Int)
\end{multline*}
and clearly,
$$
\HF_2(\Int/(\ell^n),\Int)=\HF_2^{ab}(\Int/(\ell^n),\Int).
$$
From this, the claim of the lemma follows easily.
\end{proof}

The following proposition was proved in
\cite[Proposition~2.3.7]{FK:CNCIT}
 in the case of number fields.

\begin{prop}\label{prop:vanishing in the limit}
Let $Q$ be a function field of transcendence degree $1$ over a
finite field $\FF$ and let $\ell$ be any prime. Then
$$
\varprojlim_{L} \SKTh_1(\Int_{\ell}[\Gal(L/Q)])=0.
$$
where $L$ runs through the finite Galois extensions of $Q$ in a
fixed separable closure $\algc{Q}$ of $Q$.
\end{prop}
\begin{proof}
By the same argument as in the proof of Proposition~2.3.7 in
\cite{FK:CNCIT}, it suffices to prove that
$$
\HF^2(\Gal(\algc{Q}/L),\Rat_{\ell}/\Int_{\ell})=0
$$
for any finite extension $L$ of $Q$.

If $\ell$ is different from the characteristic of $\FF$, then the
vanishing of this group can be deduced via the same argument as
the analogous statement for number fields given in \cite[$\S$ 4,
Satz~1]{Schn:Galoiskohomologiegruppen}: Let
$$
L_{\infty}=\bigcup_n L(\zeta_{\ell^n}).
$$
Then
$$
\HF^2(\Gal(\algc{Q}/L),\Rat_{\ell}/\Int_{\ell})=\HF^1(\Gal(L_{\infty}/L),\HF^1(\Gal(\algc{Q}/L_{\infty}),\Rat_{\ell}/\Int_{\ell}))
$$
and by Kummer theory,
$$
\HF^1(\Gal(\algc{Q}/L_{\infty}),\Rat_{\ell}/\Int_{\ell})=L_{\infty}^{\times}\tensor_{\Int}\Rat_{\ell}/\Int_{\ell}(-1).
$$
Now
$$
\HF^1(\Gal(L_{\infty}/L),L_{\infty}^{\times}\tensor_{\Int}\Rat_{\ell}/\Int_{\ell}(-1))=\varinjlim_{n}
\HF^1(\Gal(L_{\infty}/L),L(\zeta_{\ell^n})^{\times}\tensor_{\Int}\Rat_{\ell}/\Int_{\ell}(-1))
$$
by \cite[Proposition~1.5.1]{NSW:CohomNumFields}  and
$$
\HF^1(\Gal(L_{\infty}/L),L(\zeta_{\ell^n})^{\times}\tensor_{\Int}\Rat_{\ell}/\Int_{\ell}(-1))=
\big(L(\zeta_{\ell^n})^{\times}\tensor_{\Int}\Rat_{\ell}/\Int_{\ell}(-1)\big)_{\Gal(L_{\infty}/L)}
$$
by \emph{loc.\,cit.}, Proposition~1.6.13. Since the latter group
is a factor group of
$$
\big(L(\zeta_{\ell^n})^{\times}\tensor_{\Int}\Rat_{\ell}/\Int_{\ell}(-1)\big)_{\Gal(L_{\infty}/L(\zeta_{\ell^n}))}=0,
$$
the claim is proved.

If $\ell$ is equal to the characteristic of $\FF$, then the
cohomological $\ell$-dimension of $\Gal(\algc{Q}/L)$ is known to
be $1$ \cite[Theorem~10.1.11.(iv)]{NSW:CohomNumFields} and hence,
the second cohomology group of $\Rat_{\ell}/\Int_{\ell}$ vanishes
for trivial reasons.
\end{proof}

\section{$L$-Functions}\label{sec:L-functions}

Consider an adic ring $\Lambda$ and let $\Lambda[[T]]$ denote the
ring of power series in the indeterminate $T$ (where $T$ is
assumed to commute with every element of $\Lambda$). The ring
$\Lambda[[T]]$ is still an adic ring whose Jacobson radical
$\Jac(\Lambda[[T]])$ is generated by $\Jac(\Lambda)$ and $T$. In
particular, we conclude from Pro\-po\-si\-tion~\ref{prop:K_1 of
adic rings} that
$$
\KTh_1(\Lambda[[T]])=\varprojlim_n
\KTh_1(\Lambda[[T]]/\Jac(\Lambda[[T]])^n)
$$
is a profinite group.

Let $\FF$ be a finite field. We write $X_0$ for the set of closed
points of a scheme $X$ in $\cat{Sch}_{\FF}$. If $x\in X_0$ is a
closed point, then
$$
\algc{x}=x\times_{\Spec \FF}\Spec \algc{\FF}
$$
consists of finitely many points, whose number is given by the
degree $\deg(x)$ of $x$, i.\,e.\ the degree of the residue field
$k(x)$ of $x$ as a field extension of $\FF$. Let
$$
s_x\colon\algc{x}\mto X
$$
denote the structure map. For any complex
$$
\cmplx{\sheaf{F}}=(\cmplx{\sheaf{F}}_I)_{I\in\openideals_{\Lambda}}
$$
in $\cat{PDG}^{\cont}(X,\Lambda)$, we write
$$
\cmplx{\sheaf{F}}_{x}=(\Sect(\algc{x},s_x^*\cmplx{\sheaf{F}}_I))_{I\in\openideals_{\Lambda}}.
$$
This defines a Waldhausen exact functor
$$
\cat{PDG}^{\cont}(X,\Lambda)\mto\cat{PDG}^{\cont}(\Lambda), \qquad
\cmplx{\sheaf{F}}\mapsto \cmplx{\sheaf{F}}_{x}.
$$
Note that $\cmplx{\sheaf{F}}_{x}$ can also be written as the
product over the stalks of $\sheaf{F}$ in the points of
$\algc{x}$:
$$
\cmplx{\sheaf{F}}_{x}\isomorph\prod_{\xi\in\algc{x}}((\cmplx{\sheaf{F}}_I)_{\xi})_{I\in\Lambda}.
$$

The geometric Frobenius automorphism
$$
\Frob_{\FF}\in\Gal(\algc{\FF}/\FF)
$$
operates on $\cmplx{\sheaf{F}}_{x}$ through its
action on $\algc{x}$. Hence, it also operates on
$\ringtransf_{\Lambda[[T]]}(\cmplx{\sheaf{F}}_{x})$. Here,
$$
\ringtransf_{\Lambda[[T]]}\colon
\cat{PDG}^{\cont}(\Lambda)\mto\cat{PDG}^{\cont}(\Lambda[[T]])
$$
denotes the change of ring functor with respect to the
$\Lambda[[T]]$-$\Lambda$-bimodule $\Lambda[[T]]$, as constructed
in De\-fi\-ni\-tion~\ref{defn:change of ring functor}. The
morphism
$$
\id-\Frob_{\FF}T\colon\ringtransf_{\Lambda[[T]]}(\cmplx{\sheaf{F}}_{x})\mto\ringtransf_{\Lambda[[T]]}(\cmplx{\sheaf{F}}_{x}).
$$
is a natural isomorphism whose inverse is given by
$$
\sum_{n=0}^{\infty}\Frob_{\FF}^nT^n.
$$

\begin{defn}
The class
$$
E_x(\cmplx{\sheaf{F}},T)=[\ringtransf_{\Lambda[[T]]}(\cmplx{\sheaf{F}}_{x})\xrightarrow{\id-\Frob_{\FF}T}
\ringtransf_{\Lambda[[T]]}(\cmplx{\sheaf{F}}_{x})]^{-1}
$$
in $\KTh_1(\Lambda[[T]])$ is called the \emph{Euler factor} of
$\cmplx{\sheaf{F}}$ at $x$.
\end{defn}

One can easily verify that the Euler factor is multiplicative on
exact sequences and that
$$
E_x(\cmplx{\sheaf{F}},T)=E_x(\cmplx{\sheaf{G}},T)
$$
if the complexes $\cmplx{\sheaf{F}}$ and $\cmplx{\sheaf{G}}$ are
quasi-isomorphic. Hence, the above assignment extends to a
homomorphism
$$
E_x(-,T)\colon \KTh_0(\cat{PDG}^{\cont}(X,\Lambda))\mto
\KTh_1(\Lambda[[T]]).
$$

\begin{lem}\label{lem:Formula for Euler factor}
Let $\xi\in \algc{x}$ be a geometric point. Then
$$
E_x(\cmplx{\sheaf{F}},T)=
[\ringtransf_{\Lambda[[T]]}(\cmplx{\sheaf{F}}_\xi)
\xrightarrow{\id-\Frob_{k(x)}T^{\deg(x)}}
\ringtransf_{\Lambda[[T]]}(\cmplx{\sheaf{F}}_\xi)]^{-1}.
$$
\end{lem}
\begin{proof}
The Frobenius automorphism $\Frob_{\FF}$ induces isomorphisms
$\cmplx{\sheaf{F}}_{\Frob_{\FF}^k\xi}\isomorph\cmplx{\sheaf{F}}_{\xi}$
for $k=1,\dotsc, \deg(x)$. For $k=\deg(x)$ we have
$\Frob_{\FF}^k\xi=\xi$ and the isomorphism is given by the
operation of $\Frob_{k(x)}$ on $\cmplx{\sheaf{F}}_{\xi}$. Hence,
we may identify $\cmplx{\sheaf{F}}_x$ with the complex
$(\cmplx{\sheaf{F}}_{\xi})^{\deg(x)}$, on which the Frobenius
$\Frob_{\FF}$ acts through the matrix
$$
\begin{pmatrix}
0&\hdotsfor{2}&0&\Frob_{k(x)}\\
\id&0&\hdotsfor{2}&0\\
0&\id&0&\hdots&0\\
\vdots&\ddots&\ddots&\ddots&\vdots\\
0&\hdots&0&\id&0
\end{pmatrix}.
$$
Let $A$ be the automorphism of
$\ringtransf_{\Lambda[[T]]}((\cmplx{\sheaf{F}}_{\xi})^{\deg(x)})$
given by the matrix
$$
\begin{pmatrix}
\id&0&\hdotsfor{2}&0\\
\id T&\id&0&\dots&0\\
\id T^2&\id T&\id&0&\vdots\\
\vdots&\ddots&\ddots&\ddots&0\\
\id T^{\deg(x)-1}&\dots&\id T^2&\id T&\id
\end{pmatrix}
$$
Then $A(\id-\Frob_{\FF}T)$ corresponds to the matrix
$$
\begin{pmatrix}
\id&0&\hdots&0& -\Frob_{k(x)}T\\
0&\id&0&\hdots& -\Frob_{k(x)}T^2\\
\vdots&&\ddots&&\vdots\\
0&\hdots&0&\id& -\Frob_{k(x)}T^{\deg(x)-1}\\
0&\hdotsfor{2}&0&\big(\id-\Frob_{k(x)}T^{\deg(x)}\big)
\end{pmatrix}
$$
Moreover, we have $[A]=1$ in $\KTh_1(\Lambda[[T]])$. Hence,
$$
[\ringtransf_{\Lambda[[T]]}(\cmplx{\sheaf{F}}_{x})\xrightarrow{\id-\Frob_{\FF}T}
\ringtransf_{\Lambda[[T]]}(\cmplx{\sheaf{F}}_{x})]
=[\ringtransf_{\Lambda[[T]]}(\cmplx{\sheaf{F}}_\xi)
\xrightarrow{\id-\Frob_{k(x)}T^{\deg(x)}}
\ringtransf_{\Lambda[[T]]}(\cmplx{\sheaf{F}}_\xi)]
$$
as claimed.
\end{proof}

\begin{prop}
The infinite product
$$
\prod_{x\in X_0}E_x(\cmplx{\sheaf{F}},T)
$$
converges in the profinite topology of $\KTh_1(\Lambda[[T]])$.
\end{prop}
\begin{proof}
For each integer $m$, there exist only finitely many closed points
$x\in X_0$ with $\deg(x)<m$. If $\deg(x)\geq m$, then we conclude
from Lem\-ma~\ref{lem:Formula for Euler factor} that the image of
$E_x(\cmplx{\sheaf{F}},T)$ in $\KTh_1(\Lambda[T]/(T^m))$ is $1$.
\end{proof}

\begin{defn}
The \emph{$L$-function} of the complex $\cmplx{\sheaf{F}}$ in
$\cat{PDG}^{\cont}(X,\Lambda)$ is given by
$$
L_{\FF}(\cmplx{\sheaf{F}},T)=\prod_{x\in
X_0}E_x(\cmplx{\sheaf{F}},T)\in \KTh_1(\Lambda[[T]])
$$
\end{defn}

\begin{rem}
If $\FF'$ is a subfield of $\FF$, then Lem\-ma~\ref{lem:Formula
for Euler factor} implies that
$$
L_{\FF'}(\cmplx{\sheaf{F}},T)=L_{\FF}(\cmplx{\sheaf{F}},T^{[\FF:\FF']})\in
\KTh_1(\Lambda[[T]]).
$$
\end{rem}

\begin{rem}
If $\Lambda$ is commutative, the determinant induces an
isomorphism
$$
\det\colon \KTh_1(\Lambda[[T]])\mto \Lambda[[T]]^{\times}.
$$
In particular, we see that the $L$-function agrees with the one
defined in \cite[Fonction $L$ mod $\ell^n$]{SGA4h} in the case of
commutative adic rings.
\end{rem}

\section{The Grothendieck trace formula}\label{sec:Grothendieck
trace formula}

In this section, we will prove the Grothendieck trace formula for
our $L$-functions.

\begin{defn}
For a scheme $X$ in $\cat{Sch}^{sep}_\FF$ and a complex
$\cmplx{\sheaf{F}}$ in $\cat{PDG}^{\cont}(X,\Lambda)$ we let
$\mathcal{L}_{\FF}(\cmplx{\sheaf{F}},T)$ denote the element
$$
\big[\ringtransf_{\Lambda[[T]]}\big(\RDer_\FF\Sectc(\algc{X},\cmplx{\sheaf{F}})\big)
\xrightarrow{\id-\Frob_{\FF}
T}\ringtransf_{\Lambda[[T]]}\big(\RDer_\FF\Sectc(\algc{X},\cmplx{\sheaf{F}})\big)\big]^{-1}
$$
in $\KTh_1(\Lambda[[T]])$.
\end{defn}

\begin{thm}[Grothendieck trace formula]\label{thm:Grothendieck trace formula}
Let $\FF$ be a finite field of characteristic $p$ and let
$\Lambda$ be an adic ring such that $p$ is invertible in
$\Lambda$. Then
$$
L_\FF(\cmplx{\sheaf{F}},T)=\mathcal{L}_\FF(\cmplx{\sheaf{F}},T)
$$
for every scheme $X$ in $\cat{Sch}^{sep}_\FF$ and every complex
$\cmplx{\sheaf{F}}$ in $\cat{PDG}^{\cont}(X,\Lambda)$.
\end{thm}

We proceed by a series of lemmas, following closely along the
lines of \cite[Chapter VI, \S 13]{Milne:EtCohom}.

\begin{lem}
Let $U$ be an open subscheme of $X$ with closed complement $Z$.
Theo\-rem~\ref{thm:Grothendieck trace formula} is true for $X$ if
it is true for $U$ and $Z$.
\end{lem}
\begin{proof}
Write $j\colon U\oimm X$ and $i\colon Z\cimm X$ for the
corresponding immersions,
$$
u\colon U\mto \Spec \FF,\qquad x\colon X\mto \Spec \FF, \qquad
z\colon Z\mto \Spec \FF
$$
for the structure morphisms. Clearly,
$$
L(X,\cmplx{\sheaf{F}})=L(U,j^*\cmplx{\sheaf{F}})L(Z,i^*\cmplx{\sheaf{F}})
$$
On the other hand, we have an exact sequence
$$
\RDer x_!j_!j^*\cmplx{\sheaf{F}}\cto \RDer
x_!\cmplx{\sheaf{F}}\qto \RDer x_!i_*i^*\cmplx{\sheaf{F}}
$$
and (chains of) quasi-isomorphisms
$$
\RDer u_!j^*\cmplx{\sheaf{F}}\simeq \RDer
x_!j_!j^*\cmplx{\sheaf{F}}\qquad \RDer
z_!i^*\cmplx{\sheaf{F}}\simeq \RDer x_! i_*i^*\cmplx{\sheaf{F}}.
$$
Hence,
$$
[\RDer x_!\cmplx{\sheaf{F}}]=[\RDer u_!j^*\cmplx{\sheaf{F}}][\RDer
z_!i^*\cmplx{\sheaf{F}}]
$$
in $\KTh_0\cat{PDG}^{\cont}(\Spec \FF,\Lambda)$. The homomorphism
\begin{align*}
\KTh_0\cat{PDG}^{\cont}(\Spec \FF,\Lambda)&\mto K_1(\Lambda[[T]]),\\
[\cmplx{\sheaf{F}}]\mapsto [\ringtransf_{\Lambda[[T]]}(\Sect(\Spec
\algc{\FF},s^*\cmplx{\sheaf{F}}))&\xrightarrow{1-\Frob_\FF
T}\ringtransf_{\Lambda[[T]]}(\Sect(\Spec
\algc{\FF},s^*\cmplx{\sheaf{F}}))]^{-1}
\end{align*}
preserves this relation.
\end{proof}

Next, we prove that the formula is compatible with change of the
base field.

\begin{lem}\label{lem:compatibility with base field change}
Let $\FF'$ be a subfield of $\FF$ and $X$ a scheme in
$\cat{Sch}^{sep}_\FF$. Then
$$
\mathcal{L}_{\FF'}(\cmplx{\sheaf{F}},T)=\mathcal{L}_{\FF}(\cmplx{\sheaf{F}},T^{[\FF:\FF']}).
$$
\end{lem}
\begin{proof}
Let $r\colon \Spec \FF\mto \Spec \FF'$ be the morphism induced by
the inclusion $\FF'\subset \FF$ and write
\begin{align*}
h\colon& X\times_{\Spec \FF}\Spec \algc{\FF} \mto X,&\qquad h'\colon& X\times_{\Spec \FF'}\Spec \algc{\FF}\mto X\\
s\colon& \Spec \algc{\FF}\mto \Spec \FF,&\qquad s'\colon& \Spec
\algc{\FF}\mto \Spec \FF'
\end{align*}
for the corresponding structure morphisms. For any
$\cmplx{\sheaf{F}}$ in $\cat{PDG}^{\cont}(X,\Lambda)$, the
complexes $ r_*\RDer h_!\cmplx{\sheaf{F}}$, $\RDer r_!\RDer
h_!\cmplx{\sheaf{F}}$, and $\RDer h'_!\cmplx{\sheaf{F}}$ in
$\cat{PDG}^{\cont}(\Spec \FF',\Lambda)$ are quasi-isomorphic.
Moreover, for any complex $\cmplx{\sheaf{G}}$ in
$\cat{PDG}^{\cont}(\Spec \FF,\Lambda)$, the following diagram is
commutative:
$$
\xymatrix@R+2pc{ \Sect(\Spec
\algc{\FF},s^{*}r^*r_*\cmplx{\sheaf{G}})\ar[r]^{\isomorph}\ar[d]^{\Frob_{\FF'}}&\bigoplus\limits_{k=1}^{[\FF:\FF']}\Sect(\Spec
\algc{\FF},s^{*}\cmplx{\sheaf{G}})\ar[d]^{\left(\begin{smallmatrix}
0&\hdots&0&\Frob_{\FF}\\
\id&0&\hdots&0\\
\vdots&\ddots&\ddots&\vdots\\
0&\hdots&\id&0
\end{smallmatrix}\right)}\\
\Sect(\Spec
\algc{\FF},s^{*}r^*r_*\cmplx{\sheaf{G}})\ar[r]^{\isomorph}&\bigoplus\limits_{k=1}^{[\FF:\FF']}\Sect(\Spec
\algc{\FF},s^{*}\cmplx{\sheaf{G}}) }
$$
As in the proof of Lem\-ma~\ref{lem:Formula for Euler factor} one
concludes
\begin{align*}
\big[\ringtransf_{\Lambda[[T]]}\big(\RDer_{\FF'}\Sectc(\algc{X},\cmplx{\sheaf{F}})\big)
\xrightarrow{\id-\Frob_{\FF'}T}
\ringtransf_{\Lambda[[T]]}\big(\RDer_{\FF'}\Sectc(\algc{X},\cmplx{\sheaf{F}})\big)\big]&=\\
\big[\ringtransf_{\Lambda[[T]]}\big(\Sect(\Spec
\algc{\FF},s^{*}r^*r_*\RDer h_!\cmplx{\sheaf{F}})\big)
\xrightarrow{\id-\Frob_{\FF'}T}\ringtransf_{\Lambda[[T]]}\big(\Sect(\Spec \algc{\FF},s^{*}r^*r_*\RDer h_!\cmplx{\sheaf{F}})\big)\big]&=\\
\big[\ringtransf_{\Lambda[[T]]}\big(\Sect(\Spec
\algc{\FF},s^{*}\RDer h_!\cmplx{\sheaf{F}})\big)
\xrightarrow{\id-\Frob_{\FF}T^{[\FF:\FF']}}\ringtransf_{\Lambda[[T]]}\big(\Sect(\Spec \algc{\FF},s^{*}\RDer h_!\cmplx{\sheaf{F}})\big)\big]&=\\
\big[\ringtransf_{\Lambda[[T]]}\big(\RDer_{\FF}\Sectc(\algc{X},\cmplx{\sheaf{F}})\big)
\xrightarrow{\id-\Frob_{\FF}
T^{[\FF:\FF']}}\ringtransf_{\Lambda[[T]]}\big(\RDer_{\FF}\Sectc(\algc{X},\cmplx{\sheaf{F}})\big)\big].
\end{align*}
\end{proof}

Clearly, Theo\-rem~\ref{thm:Grothendieck trace formula} is true
for schemes of dimension $0$. Next, we consider the case that $X$
is a curve.

\begin{lem}\label{lem:smooth curve case}
The formula in Theo\-rem~\ref{thm:Grothendieck trace formula} is
true for any smooth and geometrically connected curve $X$,
$\Lambda=\Int_{\ell}[G]$, and
$\cmplx{\sheaf{F}}=\Int_{\ell}[G]^{\sharp}_{X}$, where $\ell$ is a
prime different from the characteristic of $\FF$ and $G$ is the
Galois group of a finite Galois covering of $X$.
\end{lem}
\begin{proof}
Let $Q$ be the function field of $X$ and let $F$ the function
field of a finite Galois covering of $X$. Let $d_F$ denote the
element
$$
d_F=
L(\Int_{\ell}[\Gal(F/Q)]_X^{\sharp},T)\mathcal{L}(\Int_{\ell}[\Gal(F/Q)]_X^{\sharp},T)^{-1}
$$
in $\KTh_1(\Int_{\ell}[\Gal(F/Q)][[T]])$.

Note that $d_F$ does not change if we replace $X$ by an open
subscheme of $X$. Hence, we may define $d_F$ for any finite Galois
extension $F$ of $Q$. If $F'/F$ is Galois, then $d_{F'}$ is mapped
onto $d_F$ under the canonical homomorphism
$$
\KTh_1(\Int_{\ell}[\Gal(F'/Q)][[T]])\mto
\KTh_1(\Int_{\ell}[\Gal(F/Q)][[T]]).
$$
Let $L$ be a splitting field for $\Rat_{\ell}[\Gal(F/Q)]$ and
$M\subset L[\Gal(F/Q)]$ a maximal $\Int_{\ell}$-order. By the
classical Grothendieck trace formula \cite[Fonction $L$ mod
$\ell^n$, Theorem~2.2.(a)]{SGA4h}, the image of $d_F$ under the
homomorphism
$$
\KTh_1(\Int_{\ell}[\Gal(F/Q)][[T]])\mto
\KTh_1(M[[T]])\isomorph\bigoplus_{k=1}^r\mathcal{O}_L[[T]]^{\times}
$$
is trivial; hence $d_F\in
\SKTh_1(\Int_{\ell}[\Gal(F/Q)][[T]])=\SKTh_1(\Int_{\ell}[\Gal(F/Q)])$.
From Pro\-po\-si\-tion~\ref{prop:vanishing in the limit} we
conclude $d_F=0$.
\end{proof}

\begin{lem}\label{lem:dim1 case}
The formula in Theo\-rem~\ref{thm:Grothendieck trace formula} is
true for any scheme $X$ in $\cat{Sch}^{sep}_{\FF}$ of dimension
less or equal 1, any adic ring $\Lambda$ with $p\in
\Lambda^{\times}$ and any complex $\cmplx{\sheaf{F}}$ in
$\cat{PDG}^{\cont}(X,\Lambda)$.
\end{lem}
\begin{proof}
By Pro\-po\-si\-tion~\ref{prop:K_1 of adic rings} it suffices to
consider finite rings $\Lambda$. The $\ell$-Sylow subgroups of
$\Lambda$ are subrings of $\Lambda$ and $\Lambda$ is equal to
their direct product. Since $p$ is invertible, the $p$-Sylow
subgroup is trivial. Hence, we may further assume that $\Lambda$
is a $\Int_{\ell}$-algebra for $\ell\neq p$.

Shrinking $X$ if necessary we may assume that $X$ is smooth,
irreducible curve and that $\cmplx{\sheaf{F}}$ is a strictly
perfect complex of locally constant sheaves. By replacing $\FF$
with its algebraic closure in the function field of $X$ and using
Lem\-ma~\ref{lem:compatibility with base field change}, we may
assume that $X$ is geometrically connected. By
Lem\-ma~\ref{lem:twisting lemma} and
Pro\-po\-si\-tion~\ref{prop:compatibility of change of rings} we
have
$$
\mathcal{L}(\cmplx{\sheaf{F}},T)=
\KTh_1(\ringtransf_{M\tensor_{\Int_{\ell}[G]}\Int_{\ell}[G][[T]]})(\mathcal{L}(\Int_{\ell}[G]^\sharp_X,T))
$$
for a suitable Galois group $G$ and a complex $\cmplx{M}$ in
$\Int_{\ell}[G]^{\op}$-$\cat{SP}(\Lambda)$. Likewise,
$$
L(\cmplx{\sheaf{F}},T)=\KTh_1(\ringtransf_{M\tensor_{\Int_{\ell}[G]}\Int_{\ell}[G][[T]]})(L(\Int_{\ell}[G]^\sharp_X,T)).
$$
Now the assertion follows from Lem\-ma~\ref{lem:smooth curve
case}.
\end{proof}

We complete the proof of Theo\-rem~\ref{thm:Grothendieck trace
formula} by induction on the dimension $d$ of $X$. By shrinking
$X$ if necessary we may assume that there exists a morphism
$f\colon X\mto Y$ such that $Y$ and all fibres of $f$ have
dimension less than $d$. Then
Pro\-po\-si\-tion~\ref{prop:properties of Rf_!}.(3) and the
induction hypothesis imply
\begin{align*}
\mathcal{L}(\cmplx{\sheaf{F}},T)=\mathcal{L}(\RDer
f_!\cmplx{\sheaf{F}},T) =L(\RDer f_!\cmplx{\sheaf{F}},T).
\end{align*}
Let now $y$ be a closed point of $Y$. Write $f_y\colon X_y\mto X$
for the fibre over $y$. Then
\begin{align*}
E_y(\RDer
f_!\cmplx{\sheaf{F}},T)&=\big[\ringtransf_{\Lambda[[T]]}\big(\RDer\Sectc(\algc{X}_y,f_y^*\cmplx{\sheaf{F}})\big)
\xrightarrow{\id-\Frob_{\FF}
T}\ringtransf_{\Lambda[[T]]}\big(\RDer\Sectc(\algc{X}_y,f_y^*\cmplx{\sheaf{F}})\big)\big]^{-1}\\
&=L(f_y^*\cmplx{\sheaf{F}},T)
\end{align*}
by Pro\-po\-si\-tion~\ref{prop:properties of Rf_!}.(4) and the
induction hypothesis. Since clearly
$$
L(\cmplx{\sheaf{F}},T)=\prod_{y\in
Y_0}L(f_y^*\cmplx{\sheaf{F}},T),
$$
Theo\-rem~\ref{thm:Grothendieck trace formula} follows.

\begin{rem}
The formula in Theo\-rem~\ref{thm:Grothendieck trace formula} is
also valid if $\Lambda$ is a finite field of characteristic $p$,
see \cite[Fonction $L$ mod $\ell^n$, Theorem~2.2.(b)]{SGA4h}.
However, it does not extend to general adic $\Int_p$-algebras. We
refer to \emph{loc.\,cit.}, $\S4.5$ for a counterexample.
\end{rem}

\bibliographystyle{amsalpha}
\bibliography{Literature}
\end{document}